\documentclass{article}

\usepackage{amsmath}
\usepackage{amssymb}
\usepackage{amsthm}
\usepackage{enumerate}
\usepackage{hyperref}
\usepackage{bbm}
\usepackage[left=2.5cm,top=2.8cm,right=2.5cm,bottom=2.8cm]{geometry}
\usepackage[english]{babel}
\usepackage{tikz}
\usepackage{soul}
\usepackage{xcolor}
\usepackage{booktabs}
\usepackage{tabularx}

\newtheorem{theorem}{Theorem}[section]
\newtheorem{corollary}[theorem]{Corollary}
\newtheorem{lemma}[theorem]{Lemma}
\newtheorem{proposition}[theorem]{Proposition}
\newtheorem{remark}[theorem]{Remark}
\newtheorem{example}[theorem]{Example}

\newcommand{\N}{{\ensuremath{\mathbb{N}}}}
\newcommand{\C}{{\ensuremath{\mathbb{C}}}}
\newcommand{\Cm}{{\ensuremath{\C^{m\times n}}}}
\newcommand{\Cn}{{\ensuremath{\C^{n\times m}}}}
\newcommand{\Cnn}{{\ensuremath{\C^{n\times n}}}}

\newcommand{\Ra}{{\ensuremath{\cal R}}}
\newcommand{\Nu}{{\ensuremath{\cal N}}}
\newcommand{\ind}{{\ensuremath{\rm Ind}}}
\newcommand{\ra}{{\ensuremath{\rm rk}}}
\newcommand{\wg}{{\ensuremath{\rm WG}}}

\newcommand{\ep}{{\ensuremath{\rm EP}}}
\newcommand{\gm}{{\ensuremath{\rm GM}}}
\newcommand{\wc}{{\text{\tiny \rm WC}}}
\newcommand{\rk}{{\ensuremath{\rm rk}}}

\newcommand{\odagger}{\mathrel{\text{\textcircled{$\dagger$}}}}
\newcommand{\core}{\mathrel{\text{\textcircled{$\#$}}}}
\newcommand{\weak}{\mathrel{\text{\textcircled{w}}}}

\begin{document}
\author{D.E. Ferreyra\thanks{Universidad Nacional de La Pampa, FCEyN, Uruguay 151, 6300 Santa Rosa, La Pampa,
Argentina.} $^\dag$,
F.E. Levis\thanks{Universidad Nacional de R\'io Cuarto, CONICET, FCEFQyN, RN 36 KM 601, 5800 R\'io Cuarto, C\'ordoba, Argentina. E-mail: \texttt{deferreyra@exa.unrc.edu.ar, flevis@exa.unrc.edu.ar, apriori@exa.unrc.edu.ar.}}, A.N. Priori$^\dag$, 
N. Thome\thanks{Instituto Universitario de
Matem\'atica Multidisciplinar, Universitat Polit\`ecnica de
Val\`encia, 46022 Valencia, Spain. E-mail: \texttt{njthome@mat.upv.es}.}}

\title{Extending EP matrices by means of recent generalized inverses}
\date{}
\maketitle

\begin{abstract}
It is well known that a square  complex matrix is called EP if it commutes with its Moore-Penrose inverse. In this paper, new classes of matrices which extend 
this concept are characterized. For that, we consider commutative equalities given by matrices of arbitrary index and generalized inverses recently investigated in the literature. More specifically, these classes are characterized  by expressions of type $A^mX=XA^m$, where $X$ is an outer inverse of a given complex square matrix $A$ and $m$ is an arbitrary positive integer. The relationships between the different classes of matrices are also analyzed. Finally, a picture presents an overview  of the overall studied classes.
\end{abstract}



\section{Introduction and Notations}

A square complex matrix $A$ is said to be EP (or range-Hermitian) if its range space is equal to the range space of its conjugate transpose $A^*$ or, equivalently, if the orthogonal projector onto its range coincides with the orthogonal projector onto the range of $A^*$. EP matrices were defined in 1950 by H. Schwerdtfeger \cite{Sch} and since then, this class of matrices has attracted the attention of several authors \cite{ChTi, Dj, MoDjKo, Pearl, RaMi, Tian}. 
The EP matrix class includes as special cases the wide classes of matrices such as hermitian matrices, skew-hermitian matrices, unitary matrices,  normal matrices, orthogonal projectors and, of course, nonsingular matrices.

EP matrices have been generalized in many ways. The first two  generalizations, known as bi-dagger  and bi-EP matrices, were given by Hartwig et al. in \cite{HaSp}. Several generalizations include them, such as conjugate EP matrices \cite{MeIn}, $K$-EP matrices \cite{MeKr}, weighted EP matrices \cite{Tian}, relative EP matrices \cite{FeMa}, etc.
All these classes of matrices can be expressed in terms of the Moore-Penrose inverse. 
Others generalized inverses,  such as DMP \cite{MaTh} inverses, core EP inverses \cite{PrMo}, CMP inverses \cite{MeSa} and WG inverses \cite{WaCh}, were introduced and thoroughly studied in the last years.

By using these inverses in equalities of type $A^kX=XA^k$, where $k$ is the index of $A$, in \cite{FeLeTh3}, generalizations of EP matrices gave rise to new classes of matrices, such as $k$-index EP (originally called index-EP in \cite{WaLi}), $k$-EP matrices \cite{MaRuTh}, $k$-CMP matrices, $k$-DMP matrices, dual $k$-DMP matrices, and $k$-core EP matrices.
 
In \cite{Pearl}, Pearl proved that  a square  complex matrix is EP if and only if it commutes with its Moore-Penrose inverse.
Recently, related to this result, in \cite{WaLi} the authors presented the weak group matrix 
as a square complex  matrix $A$ that commute with its WG inverse as we will recall later. This new class of matrices is wider 
than that of the well-known group matrices, given by those matrices of index at most one, to which EP matrices belong.  

Motivated  by these papers, our purpose is to give new characterizations of the aforementioned  classes of matrices in terms of only equalities of type $AX=XA$, for $X$ being an outer generalized inverse. 
Consequently, we will provide an easier condition for that to be checked.    
Since the power of $A$ will turn out not to be relevant, the main advantage is now that shall be not necessary to compute explicitly the index. 

The symbol $\Cm$ stands for the set of all $m\times n$ complex matrices. For any matrix $A\in \Cm$, we denote its conjugate transpose,
 inverse (whenever it exists),  rank,  null space, and  range space of $A$ by  $A^{\ast},$ $A^{-1}$, $\rk(A)$, $\Nu(A)$, and  $\Ra(A)$, respectively.  Moreover, $I_n$ will refer to the $n \times n$ identity matrix. The index of $A\in \Cnn$, denoted by $\ind(A)$, is the smallest nonnegative integer $k$ such that $\rk(A^k) = \rk(A^{k+1})$. Throughout this paper, we will assume that $\ind(A)=k\ge 1$.
 
We start recalling the definitions of some generalized
inverses.  For $A\in\Cm$, the Moore-Penrose inverse $A^\dag$ of $A$ is the unique matrix $X \in \Cn$ satisfying the following four equations: $AXA=A$, $XAX=X$, $(AX)^*=AX$, and $(XA)^*=XA$. The Moore-Penrose inverse can be used to represent the orthogonal projectors $P_A:=AA^{\dag}$ and $Q_A:=A^{\dag}A$ onto $\Ra(A)$ and $\Ra(A^*)$, respectively. 
A matrix $X\in \Cn$ satisfying $AXA=A$ is called an inner inverse of $A$, while a matrix $X\in \Cn$ satisfying $XAX=X$ is called an outer inverse of $A$. 

For $A\in \Cnn$, the Drazin inverse $A^d$ of $A$ is the unique matrix $X\in \Cnn$ satisfying $X A X=X$, $AX=XA$, and $XA^{k+1}=A^k$. If $\ind(A)\le 1$, then the Drazin inverse is called the group inverse of $A$ and denoted by $A^\#$.
It is known that $A^\#$ exists if and only if $\ra(A)=\ra(A^2)$, in which case $A$ is called a group matrix. 

The symbols $\C_n^{\gm}$ and $\C_n^{\ep}$ will denote the subsets of $\Cnn$ consisting of group matrices and EP matrices, respectively. It is well known  that $\C_n^{\ep}\subseteq \C_n^{\gm}$  \cite{BeGr}.
  
In 1972, Rao et al. \cite{RaMi2} introduced the (uniquely determined) generalized inverse $A^-_{\rho^*,\chi}=A^\#P_A$ of $A \in \Cnn$. 
In 2010, Baksalary et al. \cite{BaTr} rediscovered it and denote it by $A^{\core}$.
 
By exploiting this same idea of generating a generalized inverse, new generalized inverses of a matrix/operator/element in a ring were  recently defined by using known generalized inverses and projectors (idempotent elements).
In Table \ref{Generalized inverses}, a glossary with the main definitions and notations related to these inverses is added, where $A \in \Cnn$ with $\ind(A)=k$. 

\begin{table}[ht]
\begin{center}
\begin{tabularx}{\textwidth}{XXX}
 \toprule  \toprule
\textbf{Definition} &  \textbf{Name} & \textbf{Reference}    \\ \midrule
$A^{\odagger}=A^k (A^{k+1})^\dag$ & core EP inverse &  \cite{PrMo, ZhCh}   \\
$A^{d,\dagger}=A^dP_A$ & DMP inverse &  \cite{MaTh} \\
$A^{\dagger, d}=Q_A A^d$ &  dual DMP inverse  &  \cite{MaTh}  \\
$A^{c,\dagger}=Q_A A^{d,\dag}$ &  CMP  inverse & \cite{MeSa}   \\
$A^{\weak}=(A^{\odagger})^2 A$ &  WG inverse &  \cite{WaCh}  \\ 
$A^{\dag, \odagger}=Q_A A^{\odagger}$ &  MPCEP inverse  &  \cite{ChMoXu}   \\
$A^{\weak,\dag}=A^{\weak} P_A$ &  WC inverse &  \cite{FeLePrTh}   \\
$A^{\dagger,\weak}=Q_A A^{\weak}$ & dual WC inverse &  \cite{FeLePrTh} \\ \bottomrule
\end{tabularx}
\caption{Recent generalized inverses}
\label{Generalized inverses}
\end{center}
\end{table}

The matrix classes corresponding to some of aforementioned generalized inverses are listed in Table \ref{Matrix classes}.

\begin{table}[ht]
\begin{center}
\begin{tabularx}{\textwidth}{XXXX}
\toprule 
\toprule
\textbf{Notation} & \textbf{Definition} & \textbf{Name} & \textbf{Reference}    \\ \midrule
$\C_{n}^{k,iEP}$  & $A^k (A^k)^\dag  = (A^k)^\dag A^k$ & $k$-index EP matrix & \cite{WaLi}\\
$\C_n^{k,\dag}$  & $A^k A^\dag  = A^\dag A^k$ & $k$-EP matrix  & \cite{MaRuTh} \\
$\C_{n}^{k,d \dagger}$  &  $ A^k A^{d,\dagger}=A^{d,\dagger} A^k$ & $k$-DMP matrix & \cite{FeLeTh3} \\
$\C_{n}^{k,\dagger d}$   &  $A^k A^{\dagger,d}=A^{\dagger,d} A^k$ & dual $k$-DMP matrix & \cite{FeLeTh3} \\
$\C_{n}^{k, c \dagger}$   & $A^k A^{c, \dagger}=A^{c,\dagger} A^k$ &  $k$-CMP matrix & \cite{FeLeTh3} \\
$\C_{n}^{k, \odagger}$   & $A^k A^{\odagger}=A^{\odagger} A^k$ &  $k$-core EP matrix & \cite{FeLeTh3} \\  
$\C_{n}^{k,\wg}$  & $A A^{\weak}=A^{\weak} A$ & WG matrix & \cite{WaLi}\\
$\C_{n}^{m,k, \odagger}$, $m\in \N$  &  $A^m A^{\odagger}=A^{\odagger} A^m$ & $\{m,k\}$-core EP matrix & \cite{ZhChTh} \\ \bottomrule
\end{tabularx}
\caption{Matrix classes}
\label{Matrix classes}
\end{center}
\end{table}

In order to recall some relationships between these sets, we refer the reader to \cite[Theorem 3.20]{FeLeTh3} and \cite[Remark 2.1 and Theorem 4.6]{WaLi}.

Moreover, in the same paper, it was proved the following interesting characterization for $k$-index EP matrices
\begin{equation}\label{index EP matrix}
 A\in \C_{n}^{k,iEP} ~ \Leftrightarrow ~ A A^{\odagger}=A^{\odagger} A.
\end{equation}
Motivated  by \eqref{index EP matrix} and the sets 
$\C_{n}^{k,\wg}$ and $\C_{n}^{m,k, \odagger}$,  our purpose is to give new characterizations of the  classes of matrices given in Table \ref{Matrix classes} in terms  of  commutative type equalities $AX=XA$ (where only the first power of $A$ is involved), or even $A^mX=XA^m$ for any positive integer $m$, for $X$ being an outer inverse.

We will also consider the new classes of matrices introduced in Table \ref{matrix classes 2}.

\begin{table}[ht]
\begin{center}
\begin{tabularx}{\textwidth}{XXX}
\toprule 
\toprule
\textbf{Notation} & \textbf{Definition} & \textbf{Name}  \\ \midrule
$\C_{n}^{k, \weak}$  & $A^k A^{\weak}=A^{\weak} A^k$ & $k$-WG matrix \\
$\C_{n}^{k, \dag \odagger}$  & $A^k A^{\dag,\odagger}=A^{\dag,\odagger} A^k$ & $k$-MPCEP matrix   \\
$\C_{n}^{k,\weak \dag}$  &  $A^k A^{\weak,\dag}=A^{\weak,\dag} A^k$ & $k$-WC matrix  \\
$\C_{n}^{k,\dag \weak}$   &  $A^k A^{\dag, \weak}=A^{\dag, \weak} A^k$ & dual $k$-WC matrix \\ \bottomrule
\end{tabularx}
\caption{New matrix classes}
\label{matrix classes 2}
\end{center}
\end{table}

We would like to highlight that the powerful of the results obtained in this paper is that we do not need to check $A^kX=XA^k$ for $k$ being the index of $A$ (for the different generalized inverses), but only $A^mX=XA^m$ for any positive integer $m$ or even $AX=XA$, which offers much more flexibility in computational aspects. 

The remainder of this paper is  structured into five sections.  In Section 2, some preliminaries are given. In Section 3, we study new classes of matrices by expressions of type $A^mX=XA^m$, where $X$ is an outer inverse of a given complex square matrix $A$ and $m$ is an arbitrary positive integer. In particular, a new characterization of the Drazin inverse is obtained. In Section 4, we obtain  new characterizations  of matrix classes  defined in Table \ref{Matrix classes}  as well as we characterize the  matrix classes given in Table \ref{matrix classes 2} by using the core EP decomposition. The equivalence between classes $k$-index EP, $k$-core EP, $\{m,k\}$-core EP, and $k$-MPCEP are provided in Section 5. Finally, in order to illustrate the whole situation, a picture presents an overview  of the overall studied classes.

\section{Preliminaries}

As stated in \cite{Wang}, every matrix $A\in \Cnn$ can be written in its Core EP decomposition \begin{equation} \label{core EP decomposition}A= U\left[\begin{array}{cc}
T & S\\
0 & N
\end{array}\right]U^*,
\end{equation}
where $U\in \Cnn$ is unitary, $T\in \C^{t\times t}$ is nonsingular, $t:=\text{rk}(T)=\text{rk}(A^k)$, and $N\in \C^{(n-t)\times (n-t)}$ is nilpotent of index $k$.

Throughout, we consider the notations  $$\Delta:=(T T^*+ S(I_{n-t}-Q_N)S^*)^{-1}$$ and 
\begin{equation}\label{T tilde}
\tilde{T}_m:=\sum\limits_{i=0}^{m-1} T^{i} S N^{m-1-i}, \quad m\in \N.
\end{equation}

From \eqref{T tilde} and applying induction on $m$, it is easy to see the following property:
\begin{equation}\label{T tilde property}
\tilde{T}_{m+1}=\tilde{T}_m N+T^mS, \quad m\in \N.
\end{equation} 
Moreover, if $m \ge k$, then $$\tilde{T}_m=\sum\limits_{i=m-k}^{m-1} T^{i} S N^{m-1-i}= \sum\limits_{i=0}^{k-1} T^{m-k+i} S N^{k-1-i},$$ and so
\begin{equation}\label{T tilde property II}
\tilde{T}_m=T^{m-k} \tilde{T}_k.
\end{equation} 
The generalized inverses previously  introduced  can be represented by using the core EP decomposition \cite{WaCh,FeLeTh3,ChMoXu,FeLePrTh}. We summarize them in Table \ref{core EP representation}. 
Here, unitary similarity will be denoted by $\stackrel{*}\approx$. Suppose   $A\in \Cnn$  is of the form (\ref{core EP decomposition}) with $\text{Ind}(A)=k$. Then:

\begin{table}[ht]
\begin{center}
\begin{tabularx}{\textwidth}{l X}
\toprule 
\toprule
\multicolumn{2}{c}{\textbf{
Representation of generalized inverses by means of Core EP decomposition
}} 
\\ 
\midrule
 \vspace{5mm}
 $A^\dag \stackrel{*}\approx \left[\begin{array}{cc}
T^*\Delta & -T^*\Delta S N^\dag\\
(I_{n-t}-Q_N)S^*\Delta & N^\dag-(I_{n-t}-Q_N)S^*\Delta S N^\dag
\end{array}\right]$ & \qquad $A^{d} \stackrel{*}\approx \left[\begin{array}{cc}
T^{-1} & T^{-(k+1)}\tilde{T}_{k} \\
0 & 0
\end{array}\right]$  \\ 
\vspace{5mm}
$A^{\dagger, d} \stackrel{*}\approx \left[\begin{array}{cc}
T^* \Delta & T^* \Delta T^{-k} \widetilde{T}_k \\
(I_{n-t}-Q_N)S^*\Delta & (I_{n-t}-Q_N)S^*\Delta T^{-k}\widetilde{T}_k 
\end{array}\right]$  & \qquad  
$A^{d,\dagger} \stackrel{*}\approx \left[\begin{array}{cc}T^{-1} & T^{-(k+1)} \tilde{T}_{k} P_N \\0 & 0\end{array}\right]$   \\ 
\vspace{5mm}
$A^{c,\dagger} \stackrel{*}\approx \left[\begin{array}{cc}
T^* \Delta & T^* \Delta T^{-k} \tilde{T}_{k} P_N\\
(I_{n-t}-Q_N)S^*\Delta & (I_{n-t}-Q_N)S^*\Delta T^{-k} \tilde{T}_{k} P_N
\end{array}\right]$ & \qquad  $A^{\odagger} \stackrel{*}\approx \left[\begin{array}{cc}
T^{-1} & 0 \\
0 & 0
\end{array}\right]$ \\ 
\vspace{5mm}
$A^{\dag, \odagger} \stackrel{*}\approx \left[\begin{array}{cc}
T^* \Delta & 0 \\
(I_{n-t}-Q_N)S^* \Delta  & 0
\end{array}\right]$  &\qquad $A^{\weak} \stackrel{*}\approx \left[\begin{array}{cc}T^{-1} & T^{-2} S\\0 & 0\end{array}\right]$   \\ 
\vspace{5mm}
$A^{\dagger,\weak} \stackrel{*}\approx \left[\begin{array}{cc}
T^* \Delta  & T^* \Delta T^{-1} S  \\
(I_{n-t}-Q_N)S^*\Delta  & (I_{n-t}-Q_N)S^*\Delta T^{-1} S
\end{array}\right]$ & \qquad 
$A^{\weak,\dag} \stackrel{*}\approx \left[\begin{array}{cc}
T^{-1}  & T^{-1} S P_N\\
0 & 0
\end{array}\right] $   \\ 
\bottomrule
\end{tabularx} 
\caption{Representation of generalized inverses}
\label{core EP representation} 
\end{center}
\end{table}

From \eqref{core EP decomposition} and Table \ref{core EP representation} we derive the following expressions for the orthogonal projectors 

\begin{equation}\label{Q_A}
\begin{split}
P_A  & \stackrel{*}\approx \left[\begin{array}{cc}
I_t &  0 \\
0 & P_N
\end{array}\right] \text{ and } 
\\ Q_A & \stackrel{*}\approx \left[\begin{array}{cc}
T^* \Delta T & T^* \Delta S (I_{n-t}-Q_N)\\
(I_{n-t}-Q_N)S^*\Delta  T & Q_N+ (I_{n-t}-Q_N)S^*\Delta S (I_{n-t}-Q_N)
\end{array}\right].
\end{split}
\end{equation}

\section{Generalized inverses commuting with $A^m$}

In this section, we study classes of matrices by means of equalities of the form $A^mX=XA^m$, where $X$ is a $\{2\}$-inverse of $A$, and $m$ is an arbitrary positive integer. We investigate characterizations of this new type of matrices by applying the core-EP decomposition. In particular, we derive a new characterization of the Drazin inverse. 

We start with an auxiliary lemma. 
\begin{lemma}\label{lemma Z} Let  $A\in \Cnn$ be written as in (\ref{core EP decomposition}) with $\text{Ind}(A)=k$. Then $Z\in A\{2\}$ and $\Ra(Z)=\Ra(A^k)$ if and only if 
\begin{equation*}
Z=U\left[\begin{array}{cc}
T^{-1} & Z_2 \\
0 & 0
\end{array}\right]U^*,
\end{equation*}
where $Z_2$ is an arbitrary matrix of adequate size. 
\end{lemma}

\begin{proof}
It is clear that 
\begin{equation} \label{Ak}
A^k= U\left[\begin{array}{cc}
T^k & \tilde{T}_k \\
0 & 0
\end{array}\right]U^*,
\end{equation}
where $\tilde{T}_k$ is the expression (\ref{T tilde}).
Assuming that $
Z= U\left[\begin{array}{cc}
Z_1 & Z_2 \\
Z_3 & Z_4
\end{array}\right]U^*
$ is partitioned according to the blocks of \eqref{core EP decomposition}, by \cite[Lemma 4.2]{FeLeTh3} we know that $Z\in A\{2\}$ and $\Ra(Z)=\Ra(A^k)$ is equivalent to  $Z A^{k+1}=A^k$ and $\Ra(Z)\subseteq\Ra(A^k)$. From \eqref{Ak}, it clear that $Z A^{k+1}=A^k$ implies $Z_1=T^{-1}$ and $Z_3=0$. 
Moreover, by \cite[Lemma 2.5]{FeLeTh1}, we have
\begin{equation}\label{proyector}
P_{A^k}=A^k (A^k)^\dagger=U\left[\begin{array}{cc}
I_t & 0 \\
0 & 0
\end{array}\right]U^*.
\end{equation} 
Thus, as $\Ra(Z)\subseteq\Ra(A^k)$ is equivalent to $P_{A^k}Z=Z$, from \eqref{proyector}, we obtain $Z_4=0$. 
\\ Since the converse is evident, the proof is complete.
\end{proof}

\begin{proposition}\label{thm Q_AZ 2} Let  $A\in \Cnn$ with $\text{Ind}(A)=k$. Let $X=Q_AZ$, where $Z\in A\{2\}$ and $\Ra(Z)=\Ra(A^k)$. The following conditions are equivalent:
\begin{enumerate}[(a)]
\item $X=A^d$;
\item $A^m X= X A^m$, for all $m \in \N$; 
\item $A^m X= X A^m$, for some $m \in \N$. 
\end{enumerate}
\end{proposition}

\begin{proof}
$(a) \Rightarrow (b) \Rightarrow (c)$ are trivial.
\\  $(c) \Rightarrow (a)$ Let $s\in \mathbb{N}$ be such that $sm\geq k$. Premultiplying by $A^{m}$ both sides of the equality $A^m X= X A^m$ we have $$A^{2m} X=A^{m}A^m X= A^{m}X A^m=X A^m A^m=X A^{2m}.$$ Following in this way, we arrive at $A^{sm}X=XA^{sm}$. \\
Since $X=Q_AZ$, $Z\in A\{2\}$, and $\Ra(Z)=\Ra(A^k)$, from \eqref{Q_A} and Lemma \ref{lemma Z} we have
\begin{equation}\label{Q_AZ} 
X = U \left[\begin{array}{cc}
T^* \Delta  & T^* \Delta  T Z_2 \\
(I_{n-t}-Q_N)S^* \Delta  & (I_{n-t}-Q_N)S^* \Delta T Z_2
\end{array}\right]
U^*:=U \left[\begin{array}{cc}
X_1  & X_2 \\
X_3  & X_4
\end{array}\right]
U^*.
\end{equation} 
 It is clear that 
\begin{equation} \label{Am}
A^{sm}= U\left[\begin{array}{cc}
T^{sm} & \tilde{T}_{sm}\\
0 & N^{sm}
\end{array}\right]U^*=U\left[\begin{array}{cc}
T^{sm} & \tilde{T}_{sm}\\
0 & 0
\end{array}\right]U^*,
\end{equation}
where $\tilde{T}_m$ has been defined in  (\ref{T tilde}).
\\ By using \eqref{Q_AZ} and \eqref{Am}, it is easy to see that $A^{sm} X= X A^{sm}$  if and only if the following conditions simultaneously hold:
\\(i) $T^{sm} X_1+\tilde{T}_{sm} X_3=X_1T^{sm}$,
\\(ii) $T^{sm} X_2+\tilde{T}_{sm} X_4=X_1 \tilde{T}_{sm}$, 
\\(iii) $0=X_3T^{sm}$,
\\(iv) $0=X_3\tilde{T}_{sm}$.
\\ From (iii) and nonsingularity of $T$ we deduce that $X_3=0$. Thus,  \eqref{Q_AZ} implies $(I_{n-t}-Q_N)S^* \Delta=0$, and therefore $\Delta^{-1}=TT^*$, 
$X_1=T^{-1}$ and $X_4=0$. So,  from (ii)
\begin{equation*} 
X = U\left[\begin{array}{cc}
T^{-1} & X_2\\
0 & 0
\end{array}\right]U^*,
\end{equation*}
where $ X_2=T^{-(sm+1)} \tilde{T}_{sm}$.  According to (\ref{T tilde property II}) we have  $$X_2=T^{-(k+1)}\tilde{T}_k.$$ Hence, $X=A^d$ by using Table \ref{core EP representation}.
\end{proof}

Similarly, one can prove the following result. 

\begin{proposition} \label{teo caract} Let  $A\in \Cnn$ with $\text{Ind}(A)=k$. Let $X\in A\{2\}$ be such that $\Ra(X)= \Ra(A^k)$. The following conditions are equivalent:
\begin{enumerate}[(a)]
\item $X=A^d$;
\item $A^m X= X A^m$, for all $m \in \N$; 
\item $A^m X= X A^m$, for some $m \in \N$. 
\end{enumerate}
\end{proposition}

\begin{proof} 
$(a) \Rightarrow (b) \Rightarrow (c)$ are trivial.
\\  $(c) \Rightarrow (a)$ Let $s\in \mathbb{N}$ be such that $sm\geq k$. As in the proof of Proposition \ref{thm Q_AZ 2} we have  $A^{sm}X=XA^{sm}$. According to Lemma  \ref{lemma Z} we obtain
\begin{equation*}
X=U\left[\begin{array}{cc}
T^{-1} & X_2 \\
0 & 0
\end{array}\right]U^*,
\end{equation*}
where $X_2$ is an arbitrary matrix of adequate size. 
By using \eqref{Am}, it is easy to see that $A^{sm} X= X A^{sm}$  if and only if $ X_2=T^{-(sm+1)}\tilde{T}_{sm}$. Therefore,  (\ref{T tilde property II}) and Table \ref{core EP representation}  complete the proof. 
\end{proof}

\begin{proposition} \label{cor ZP_A} 
Let  $A\in \Cnn$ with $\text{Ind}(A)=k$. Let $X=ZP_A$, where $Z\in A\{2\}$ and $\Ra(Z)=\Ra(A^k)$. The following conditions are equivalent:
\begin{enumerate}[(a)]
\item $X=A^d$;
\item $A^m X= X A^m$, for all $m \in \N$; 
\item $A^m X= X A^m$, for some $m \in \N$.
\end{enumerate}
\end{proposition}

\begin{proof}
Let $X=ZP_A$. Note that $ZAZ=Z$ implies $XAX=X$. In consequence, as $\Ra(Z)=\Ra(A^k)$ we obtain \[\Ra(X)=\Ra(XA)=\Ra(ZP_AA)=\Ra(ZA)=\Ra(Z)=\Ra(A^k).\] Now, the result follows from Proposition \ref{teo caract}.
\end{proof}

\begin{theorem} \label{k-commutative} Let  $A\in \Cnn$ with $\text{Ind}(A)=k$. 
For each generalized inverse
$$X\in \{A^{\odagger}, A^{d,\dagger}, A^{\weak}, A^{\weak,\dag}, A^{\dagger, d}, A^{c,\dagger}, A^{\dag, \odagger}, A^{\dag, \weak}\},$$  the following conditions are equivalent:
\begin{enumerate}[(a)]
\item $X=A^d$;
\item $A^m X= X A^m$, for all $m \in \N$; 
\item $A^m X= X A^m$, for some $m \in \N$.
\end{enumerate}
\end{theorem}

\begin{proof}
If $X=A^{\odagger}$,  it is well known that $X$ is an outer inverse of $A$ and its range space is $\Ra(A^k)$ \cite{PrMo}. So, by Proposition \ref{teo caract}, we have the required equivalence. \\
If $X=A^{d,\dag}$, by using Table \ref{Generalized inverses} we have $X=Z P_A$ with   $Z=A^d$. Clearly, $A^d\in A\{2\}$ and $\Ra(A^d)=\Ra(A^k)$. So, the result follows from Proposition \ref{cor ZP_A}.\\
If $X=A^{\weak}$, by \cite[Remark 3.5]{WaCh} we know that $A^{\weak}AA^{\weak}=A^{\weak}$ and $\Ra(A^{\weak})=\Ra(A^k)$. Thus, Proposition \ref{teo caract} implies the equivalence of conditions (a)-(c). \\
If $X=A^{\weak, \dag}$, from Table \ref{Generalized inverses} we have $X=Z P_A$ with   $Z=A^{\weak}$. 
As mentioned above,
$A^{\weak}\in A\{2\}$ and $\Ra(A^{\weak})=\Ra(A^k)$. Thus, Proposition  \ref{cor ZP_A} gives the result. \\
Similarly, for $X\in \{A^{\dag,d}, A^{c,\dagger}, A^{\dag, \odagger}, A^{\dag, \weak}\}$,  from Table \ref{Generalized inverses} and Proposition \ref{thm Q_AZ 2} we can obtain the required equivalences.  
\end{proof}

Other important consequence from  Proposition \ref{teo caract} is a new characterization of the Drazin inverse (and the particular case for the group inverse), which closes this section.

\begin{theorem}\label{caracdrazin}
Let  $A\in \Cnn$ with $\text{Ind}(A)=k$.  The following conditions are equivalent:
\begin{enumerate}[(a)]
\item $X=A^d$;
\item $AX^2=X$, $A^kX=XA^k$,  $\Ra(A^k) \subseteq \Ra(X)$.
\end{enumerate} 
\end{theorem}

\begin{proof} $(a) \Rightarrow (b)$ is obvious.
\\ $(b)  \Rightarrow (a)$ Suppose that 
\begin{equation} \label{f1} A^k X=X A^k.
\end{equation} We claim that $XAX=X$. \\ Indeed, if $k=1$, it is obvious. We assume $k>1$. Post-multiplying (\ref{f1}) by $X$ we have $ A^k X^2=X A^k X$. Since $AX^2=X$ we get $$XA^kX=A^k X^2=A^{k-1}AX^2=A^{k-1}X.$$ Again, post-multiplying this last equality by $X$, we get $XA^kX^2=A^{k-1}X^2$, and we deduce that $XA^kX^2=A^{k-2}AX^2= A^{k-2}X$. 
On the other hand, we also have $XA^kX^2=X A^{k-1}A X^2=XA^{k-1}X$. Then, 
$$X A^{k-1} X= A^{k-2}X.$$ 
Following this way, we derive $XA^2X=AX$. Post-multiplying by $X$ once again, we deduce that  $XA^2X^2=AX^2= X$. Thus, 
\begin{equation} \label{f2} X= X A^2 X^2= X A A X^2=XAX.
\end{equation}
As $\Ra(A^k) \subseteq \Ra(X)$, from \cite[Lemma 4.1]{FeLeTh3}
we deduce that  
\begin{equation} \label{f3} {\Ra}(X)={\Ra}(A^k).
\end{equation} Finally (\ref{f1})-(\ref{f3}) and Proposition \ref{teo caract} imply $X=A^d$.
\end{proof}

\begin{corollary}
Let $A\in \Cnn$. Then $X=A^\#$ if and only if $AX^2=X$, $AX=XA$, $\Ra(A) \subseteq \Ra(X)$. 
\end{corollary}

\section{New characterizations of recent classes of matrices}

In this section, we obtain  new characterizations  of  $k$-EP matrices (or equivalently $k$-CMP matrices), $k$-DMP matrices, and dual $k$-DMP matrices, as well as we investigate the sets of  $k$-WC  and dual $k$-WC matrices.

The set of $k$-DMP matrices was characterized in \cite{FeLeTh3}  as follows:
\begin{equation} \label{k DMP matrices}
A\in \C_{n}^{k,d\dag}  \Leftrightarrow A^{d,\dag}=A^d \Leftrightarrow   A^{c,\dag}=A^{\dag, d} \Leftrightarrow \Nu(N^*)\subseteq \Nu(\widetilde{T}_k),
\end{equation}
provided that $A$ is written as in (\ref{core EP decomposition}).
Next, we give some more characterizations of $k$-DMP matrices. 
\begin{theorem} \label{class dmp}  Let  $A\in \Cnn$ be written as in (\ref{core EP decomposition}) with $\text{Ind}(A)=k$. The following conditions are equivalent:
\begin{enumerate}[(a)]
\item $A\in \C_{n}^{k,d \dagger}$;
\item $A^m A^{d,\dagger}=A^{d,\dagger}A^m$, for all $m\in \N$;
\item $A^m A^{d,\dagger}=A^{d,\dagger}A^m$, for some $m\in \N$;
\item $A^{d,\dagger}=A^d$;
\item $A^{k+1} A^\dag = A^k$.
\end{enumerate}
\end{theorem}

\begin{proof}
$(b) \Leftrightarrow (c)\Leftrightarrow (d)$ It directly follows from Theorem  \ref{k-commutative} with $X=A^{d,\dag}$. \\
$(a)\Rightarrow (c)$ and $(d)\Rightarrow (a)$ are clear.
\newline $(a)\Rightarrow(e)$ Assume $A^{k} A^{d,\dag}=A^{d,\dag} A^k$. From definitions of Moore-Penrose and Drazin inverses, and by using Table \ref{Generalized inverses} we have 
\begin{equation*}
A^{k} A^{\dag} =  A^{d} A^{k+1} A^{\dag} = A^{k} (A^{d} A A^{\dag})  = (A^{d} A A^{\dag}) A^{k} = A^{d} A^{k}.
\end{equation*}
Now, multiplying on the left the above equation by $A$ and by using the properties $A^d A=AA^d$ and $A^d A^{k+1}=A^k$, item $(e)$ follows.
\newline  $(e)\Rightarrow(a)$ If $A^{k+1} A^{\dag} = A^{k}$, in the same manner as above, we get 
\begin{equation*}
A^{k} A^{d,\dag}=A^{k} A^{d} A A^{\dag}= A^{d} A^{k+1} A^{\dag}= A^{d} A^{k}=A^{d} A A^{\dag} A^{k}=A^{d,\dag}A^{k}. 
\end{equation*} 
\end{proof}

The set of dual $k$-DMP matrices was characterized in \cite{FeLeTh3}  as follows:
\begin{equation} \label{dual k DMP matrices}
A\in \C_{n}^{k,\dag d}  \Leftrightarrow A^{\dag,d}=A^d \Leftrightarrow   A^{c,\dag}=A^{d,\dag} \Leftrightarrow \Nu(N)\subseteq \Nu(S),
\end{equation}
provided $A$ is written as in (\ref{core EP decomposition}). 
Next, we present new characterizations for dual $k$-DMP matrices. 
\begin{theorem} \label{class dual dmp}   
Let  $A\in \Cnn$ be written as in (\ref{core EP decomposition}) with $\text{Ind}(A)=k$. The following conditions are equivalent:
\begin{enumerate}[(a)]
\item $A\in \C_{n}^{k,\dagger d}$;
\item $A^m A^{\dagger,d}=A^{\dagger,d} A^m$, for all $m\in \N$;
\item $A^m A^{\dagger,d}=A^{\dagger,d} A^m$, for some $m\in \N$;
\item $A^{\dagger,d}=A^d$;
\item $A^\dag A^{k+1} = A^k$.
\end{enumerate}
\end{theorem}

\begin{proof}
$(b) \Leftrightarrow (c) \Leftrightarrow (d)$  follow directly from Theorem \ref{k-commutative} with $X=A^{\dag,d}$. \\
$(a)\Rightarrow (c)$ and $(d)\Rightarrow (a)$ are clear.
\newline $(a)\Leftrightarrow(e)$ can be shown analogously to the proof of Theorem \ref{class dmp}.
\end{proof}

By \cite{FeLeTh3}, Theorem \ref{class dmp} and Theorem \ref{class dual dmp}, we can to obtain easily  the following  characterizations for $k$-EP matrices.

\begin{theorem} \label{class p-EP}  
Let  $A\in \Cnn$ be written as in (\ref{core EP decomposition}) with $\text{Ind}(A)=k$. The following conditions are equivalent:
\begin{enumerate}[(a)]
\item $A \in \C_n^{k,\dag}$;
\item $A^m X = X A^m$, for all $m\in \N$,  where $X\in \{A^{d,\dagger},A^{\dagger,d}\}$; 
\item $A^m X = X A^m$, for some $m\in \N$,  where $X\in \{A^{d,\dagger},A^{\dagger,d}\}$; 
\item $X=A^d$, where $X \in \{A^{d,\dagger},A^{\dagger,d}\}$;
\item $A^{k+1} A^\dag = A^\dag A^{k+1} =A^k$.
\end{enumerate}
\end{theorem}

In \cite{WaLi},  the following equivalences were proved:  
\begin{equation*} 
A\in \C_{n}^{k, \wg}  \Leftrightarrow (A^2)^{\weak}=(A^{\weak})^2  \Leftrightarrow A^{\weak}=A^d \Leftrightarrow  SN=0,
\end{equation*}
provided that $A$ is written as in (\ref{core EP decomposition}). 

Recently, in \cite{ZhChTh},   some more characterizations of WG matrices were obtained by using the core-EP decomposition.   More precisely, the authors proved the following equivalences:  
\begin{equation} \label{WG equivalences 2}
\begin{split}
A\in \C_{n}^{k, \wg} &  \Leftrightarrow A^m A^{\weak} = A^{\weak} A^m \,\,\text{for  arbitrary} \,\, m\in \N  \\ & \Leftrightarrow \tilde{T}_m N=0 \,\,\text{for arbitrary} \,\, m\in \N.
\end{split}
\end{equation}

Now, another characterizations of WG matrices are analysed.

\begin{theorem} \label{class weak}  Let  $A\in \Cnn$ be written as in (\ref{core EP decomposition}) with $\text{Ind}(A)=k$. The following conditions are equivalent:
\begin{enumerate}[(a)]
\item $A\in \C_{n}^{k, \wg} $;
\item $A^{d,\dag}=A^{\odagger}$;
\item $A^{\dag,\odagger}=A^{c,\dag}$;
\item $A^{\odagger} A^2 = A^{\odagger} A^2 A^{\odagger} A $;
\item $\Ra(N) \subseteq \Nu(\tilde{T}_k)$. 
\end{enumerate}
\end{theorem}

\begin{proof}
$(a)\Leftrightarrow (d)$ In \cite[Theorem 4.4]{WaLi}, it was proved that $A$ is a WG matrix if and only if $A^{\odagger} A^2$  commute with $A^{\odagger} A$. Now, the result follows by using that $A^{\odagger}$ is an outer inverse of $A$. \\
$(b)\Leftrightarrow (c) \Leftrightarrow (e)$  By using Table \ref{core EP representation},  we have that $A^{d,\dag}=A^{\odagger}$ is equivalent to $T^{-(k+1)} \tilde{T}_k NN^\dag=0$ which in turn is equivalent to $\tilde{T}_k N=0$ because   $T$ is nonsingular and $N^\dag$ is an inner inverse of $N$. Clearly, $\tilde{T}_k N=0$ if and only if $\Ra(N) \subseteq \Nu(\tilde{T}_k)$.  According to Table \ref{core EP representation} we have that the equality $A^{\dag,\odagger}=A^{c,\dag}$ is equivalent to $\tilde{T}_k N=0$.\\
$(a)\Leftrightarrow (e)$ follows directly  from \eqref{WG equivalences 2} for $m=k$.
\end{proof}

The case for $k$-WC matrices is studied in the following result.
\begin{theorem} \label{class 4.5}
 Let  $A\in \Cnn$ be written as in (\ref{core EP decomposition}) with $\text{Ind}(A)=k$. The following conditions are equivalent:
\begin{enumerate}[(a)]
\item $A \in \C_{n}^{k,\weak \dag}$;
\item $A^m A^{\weak,\dagger}=A^{\weak,\dagger}A^m$, for all $m\in \N$;
\item $A^m A^{\weak,\dagger}=A^{\weak,\dagger}A^m$, for some $m\in \N$;
\item $A^{\weak,\dagger}=A^d$;
\item $A \in  \C_{n}^{\wc} \cap \C_{n}^{k,d \dagger}$;
\item $\Ra(N^2)\subseteq \Nu(S)$  and $\Nu(N^*) \subseteq \Nu(\tilde{T}_k)$;
\item $SN+TS(I_{n-t}-P_N)=0$.
\end{enumerate}
\end{theorem}

\begin{proof}  $(b) \Leftrightarrow (c) \Leftrightarrow (d)$ It follows from  Theorem \ref{k-commutative} with $X=A^{\weak, \dag}$. \\
$(a)\Rightarrow (c)$ and $(d)\Rightarrow (a)$ are clear.\\ 
$(e)\Leftrightarrow (f)$ By \cite[Theorem 5.5]{FeLePrTh} and \cite[Theorem 3.13]{FeLeTh3}. \\
$(f)\Rightarrow (g)$ Clearly $\Ra(N^2)\subseteq \Nu(S)$  and $\Nu(N^*) \subseteq \Nu(\tilde{T}_k)$ are equivalent to  $SN^2=0$ and $\tilde{T}_k(I_{n-t}-P_N)=0$, respectively. So, from \eqref{T tilde} we obtain 
\[
0=\tilde{T}_k(I_{n-t}-P_N)=T^{k-2}(SN +TS(I_{n-t}-P_N)),\]
whence $SN+TS(I_{n-t}-P_N)=0$ because $T$ es nonsingular. \\
$(g)\Rightarrow (f)$  
Postmultiplying $SN+TS(I_{n-t}-P_N)=0$ by $N$ we immediately get $SN^2=0$, hence $\Ra(N^2)\subseteq \Nu(S)$. By using this fact, the equality $\tilde{T}_k(I_{n-t}-P_N)=0$ can be obtained as before. \\
$(d) \Leftrightarrow (f)$ By \cite[Theorem 6.1 (b)]{FeLePrTh}.
\end{proof}
\begin{remark}{\rm  Note that if $k\le 2$ then $\C_{n}^{k,d \dag}=\C_{n}^{k,\weak \dag}$. In fact,  if  $A\in \Cnn$ is written as in (\ref{core EP decomposition}), it is clear that $SN^2=0$ (or equivalently $\Ra(N^2)\subseteq \Nu(S)$) always is true. Now, the assertion holds by \eqref{k DMP matrices} and Theorem \ref{class 4.5} (f). Moreover, in this case, $A^{d,\dag}=A^{\weak,\dag}=A^d$ and $A^{c,\dag}=A^{\dag, d} $.
}
\end{remark}

Finally, next result provides characterizations for dual $k$-WC matrices.

\begin{theorem} \label{class 4.6}
 Let  $A\in \Cnn$ be written as in (\ref{core EP decomposition}) with $\text{Ind}(A)=k$. The following conditions are equivalent:
\begin{enumerate}[(a)]
\item $A \in \C_{n}^{k,\dag \weak}$;
\item $A^m A^{\dagger,\weak}=A^{\dagger, \weak}A^m$, for all $m\in \N$;
\item $A^m A^{\dagger, \weak}=A^{\dagger, \weak}A^m$, for some $m\in \N$;
\item $A^{\dagger, \weak}=A^d$;
\item $A \in  \C_{n}^{k, \wg} \cap \C_{n}^{k, \dagger d}$;
\item $\Ra(N)\subseteq \Nu(S)$ and $\Nu(N) \subseteq \Nu(S)$;
\item $SN +TS(I_{n-t}-Q_N)=0$.
\end{enumerate}
\end{theorem}

\begin{proof}  $(b) \Leftrightarrow (c) \Leftrightarrow (d)$ follows from  Theorem \ref{k-commutative}  with $X=A^{\dag,\weak}$. \\
$(a)\Rightarrow (c)$ and $(d)\Rightarrow (a)$ are clear.\\  
$(e)\Leftrightarrow (f)$  is a consequence of \eqref{WG equivalences 2} and   \cite[Theorem 3.16]{FeLeTh3}.\\
$(b)\Leftrightarrow (f)$  Firstly, by using Table \ref{Generalized inverses}  note that $A A^{\dag, \weak}=AA^{\weak}$ holds. Thus, from Table \ref{core EP representation} we have that $A A^{\dag, \weak} = A^{\dag, \weak} A$ is equivalent to 
\begin{equation*}
\left[\begin{array}{cc}
I_t & T^{-1}  S  \\
0  & 0
\end{array}\right]
 =\left[\begin{array}{cc}
T^* \Delta T  & T^* \Delta S+  T^* \Delta T^{-1} S N  \\
(I_{n-t}-Q_N)S^* \Delta T  & (I_{n-t}-Q_N)S^* \Delta S+(I_{n-t}-Q_N)S^* \Delta T^{-1} S N 
\end{array}\right],
\end{equation*}
which in turn is equivalent to $(I_{n-t}-Q_N)S^*=0$ and $SN=0$, because $T$ and $\Delta$ are nonsingular. Clearly, $(I_{n-t}-Q_N)S^*=0$ and $SN=0$ are equivalent to $\Nu(N) \subseteq \Nu(S)$ and $\Ra(N)\subseteq \Nu(S)$, respectively. Now,  from $(b)\Leftrightarrow (c)$ we obtain the desired equivalence. \\
$(f)\Rightarrow (g)$ As before, $(f)$ is equivalent to  $SN=0$ and $S(I_{n-t}-Q_N)=0$, respectively. So, 
$SN +TS(I_{n-t}-Q_N)=0$. \\
$(g)\Rightarrow (f)$ Assume  $SN +TS(I_{n-t}-Q_N)=0$. Post-multiplying by $Q_N$ we have $SN=0$, and consequently  $S(I_{n-t}-Q_N)=0$ because $T$ is nonsingular.  Thus, $(f)$ holds.
\end{proof}

\section{\textit{k}-index EP, \textit{k}-core EP, \{\textit{m,k}\}-core EP, and \textit{k}-MPCEP matrices}

In 2018, Ferreyra et al. \cite{FeLeTh3} showed that $A$ is $k$-core EP if and only if the core EP inverse and Drazin inverse coincide.  More precisely, they gave the following characterization by using the core-EP decomposition of $A$ given in \eqref{core EP decomposition}:
\begin{equation}\label{equality classes 1}
A\in\C_{n}^{k, \odagger}  \Leftrightarrow A^{\odagger}=A^d \Leftrightarrow \widetilde{T}_k=0.
\end{equation}

In \cite{WaLi}, the authors proved that if $A$ is written as \eqref{core EP decomposition}, then 
\begin{equation}\label{equality classes 2}
A\in \C_{n}^{k,iEP}  \Leftrightarrow  A^{\odagger}=A^d \Leftrightarrow S=0.
\end{equation}

Also recall that the set of $\{m,k\}$-core EP matrices given in Table \ref{Matrix classes} was introduced in order to extend the set of $k$-core EP matrices. In \cite{ZhChTh}, the authors derived the following characterization by using  core-EP decomposition:
\begin{equation}\label{equality classes 3}
A\in \C_{n}^{m,k, \odagger}  \Leftrightarrow \widetilde{T}_m=0,
\end{equation}
where $\widetilde{T}_m$ is as in \eqref{T tilde}.
Note that if $m=k$, the concepts of $\{m,k\}$-core EP matrix and $k$-core EP matrix coincide. 
Next, we will show that both notions are always equivalent even for $m\ne k$.  In particular, we derive that the classes given above  are equivalent to the class of  $k$-MPCEP matrices given in  Table \ref{matrix classes 2}. 

\begin{theorem}\label{class i-EP2}
 Let  $A\in \Cnn$ be written as in (\ref{core EP decomposition}) with $\text{Ind}(A)=k$. The following conditions are equivalent:
\begin{enumerate}[(a)]
\item $A \in \C_{n}^{k,\odagger}$;
\item $A\in \C_{n}^{k,iEP}$;
\item $A\in \C_{n}^{m,k, \odagger}$, where $m$ is an arbitrary positive integer;
\item $A \in \C_{n}^{k, \dag \odagger}$;
\item $A\in \C_n^{k,\weak \dag}\cap \C_n^{k, \dag \weak}$;
\item $A^m A^{\odagger}=A^{\odagger} A^m$, for all $m\in \N$;
\item $A^m A^{\odagger}=A^{\odagger}A^m$, for some $m\in \N$;
\item $A^{\odagger}=A^d$; 
\item $A^m A^{\dag, \odagger}=A^{\dag, \odagger} A^m$, for all $m\in \N$;
\item $A^m A^{\dag, \odagger}=A^{\dag, \odagger}A^m$, for some $m\in \N$;
\item $A^{\dag, \odagger}=A^d$; 
\item $S=0$;
\item $\tilde{T}_k=0$;
\item $\tilde{T}_m=0$, where $m$ is an arbitrary positive integer;
\item $A^{\odagger}=A^d=A^{d,\dagger}=A^{\dagger,d}= A^{c,\dagger}=A^{\weak}= A^{\dag, \odagger}=A^{\weak,\dag}=A^{\dag, \weak}$.
\end{enumerate}
\end{theorem}

\begin{proof}
$(a)\Leftrightarrow (b) \Leftrightarrow (h)\Leftrightarrow (l) \Leftrightarrow (m)$  follows from \eqref{equality classes 1} and \eqref{equality classes 2}. \\
 $(a)\Leftrightarrow (c) \Leftrightarrow (f)\Leftrightarrow (g) \Leftrightarrow (n)$ is a consequence from \eqref{equality classes 3} and Theorem \ref{k-commutative} with $X=A^{\odagger}$. \\ 
 $(i) \Leftrightarrow (j) \Leftrightarrow (k)$ 
 follows from Theorem \ref{k-commutative} with $X=A^{\dag, \odagger}$.
 \\ 
 $(d)\Rightarrow (j)$ and $(k)\Rightarrow (d)$ are clear.\\  
 $(e) \Leftrightarrow (l)$ is a consequence of items (a) and (g) in Theorem \ref{class 4.5} and items (a), (f) and (g) in Theorem \ref{class 4.6}.\\
$(i) \Rightarrow (l)$ Let $m=1$. From Table \ref{core EP representation},  it is easy to see  that  $AA^{\dag, \odagger}=A^{\dag, \odagger} A$  implies $T^* \Delta S=0$,  which is equivalent to $S=0$ because $T$ and $\Delta$ are nonsingular. \\
$(l)\Rightarrow (k)$ If $S=0$, from Table \ref{core EP representation}  we have $A^{\dag, \odagger}=A^d$. \\
$(k)\Rightarrow (o)$ Since (k) is equivalent to (l) we have $S=0$. Now, the assertion follows from \eqref{core EP decomposition} and Table \ref{core EP representation}.  \\
$(o)\Rightarrow (k)$ Trivial.
 \end{proof}
 
Note that if $A\in \Cnn$ is written as in (\ref{core EP decomposition}), it is easy to see that $A\in \C_n^{\ep}$ if and only if $S=0$ and $N=0$. Now, from \eqref{core EP decomposition}, Table \ref{core EP representation}, and Theorem \ref{class i-EP2} we have the following corollary.

\begin{corollary} \label{class i-EP}   Let  $A\in \Cnn$. If $ A\in \C_n^{\ep}$ then   \[A^\dag=A^{\odagger}=A^d=A^{d,\dagger}=A^{\dagger,d}= A^{c,\dagger}=A^{\weak}= A^{\dag, \odagger}=A^{\weak,\dag}=A^{\dag, \weak}.\]  Conversely, if $X=A^\dag$ for some $X\in \{ A^{\odagger}, A^d, A^{d,\dagger}, A^{\dagger,d}, A^{c,\dagger}, A^{\weak}, A^{\dag, \odagger}, A^{\weak,\dag}, A^{\dag, \weak}\}$ then $A\in \C_n^{\ep}$.
\end{corollary}

The following examples show that the class $\C_{n}^{k, iEP}$ is strictly included in $\C_n^{k,\weak \dag}$ and $\C_n^{k,\dag \weak}$, which in turn are included  in $\C_n^{k,d \dag}$ and $\C_n^{k,\dag d}$, respectively.

\begin{example} \label{example 1} 
Consider the matrix
\[
A= \left[\begin{array}{rrrr}
1 &  0 & -1 & 1\\
0 &  0 & -1 & 0  \\
0 &  0 &  0 & 1 \\
0 & 0 & 0 & 0
\end{array}\right].
\]
Clearly, $\ind(A)=3$. Denoting $T=1$, 
$S=\left[\begin{array}{rrr}
0 &  -1 & 1 
\end{array}\right]$  and $N= \left[\begin{array}{rrr}
0 &  -1 & 0  \\
0 &  0 & 1  \\
0 & 0 & 0\\
\end{array}\right]$, we have $SN =\left[\begin{array}{rrr}
0 & 0 & -1 
\end{array}\right]$ and 
$P_N=NN^\dag= \left[\begin{array}{rrr}
1 &  0 & 0  \\
0 &  1 & 0  \\
0 & 0 & 0\\
\end{array}\right]$. Therefore, 
\[SN+TS(I_{n-t}-P_N)=\left[\begin{array}{rrr}
0 &  0 & -1
\end{array}\right]+\left[\begin{array}{rrr}
0 &  0 & 1 
\end{array}\right]=\left[\begin{array}{rrr}
0 &  0 & 0 
\end{array}\right],\] and consequently  $A\in \C_n^{k,\weak \dag}$ by Theorem \ref{class 4.5}. However, since $S\neq 0$, according to  Theorem \ref{class i-EP2} we have $A\notin \C_{n}^{k, iEP}$.
\end{example} 

\begin{example} \label{example 2} 
Let 
\[
A= \left[\begin{array}{rrrr}
1 &  0 & 0 & 1\\
0 &  0 & 0 & 1  \\
0 &  0 &  0 & 0 \\
0 & 0 & 0 & 0
\end{array}\right].
\]
It is easy to see that $\ind(A)=2$. Since $T=1$, 
$S=\left[\begin{array}{rrr}
0 &  0 & 1 
\end{array}\right]$  and $N= \left[\begin{array}{rrr}
0 &  0 & 1  \\
0 &  0 & 0  \\
0 & 0 & 0\\
\end{array}\right]$, we have $SN =\left[\begin{array}{rrr}
0 & 0 & 0 
\end{array}\right]$ and 
$Q_N=N^\dag N= \left[\begin{array}{rrr}
0 &  0 & 0  \\
0 &  0 & 0  \\
0 & 0 & 1\\
\end{array}\right]$. Thus, 
\[SN+TS(I_{n-t}-Q_N)=\left[\begin{array}{rrr}
0 &  0 & 0
\end{array}\right]+\left[\begin{array}{rrr}
0 &  0 & 0
\end{array}\right]=\left[\begin{array}{rrr}
0 &  0 & 0 
\end{array}\right]\] and so  $A\in \C_n^{k,\dag \weak}$ by Theorem \ref{class 4.6}. However,  from Theorem \ref{class i-EP2} we get $A\notin \C_{n}^{k, iEP}$. 
\end{example} 
\begin{example} \label{example 3} The class $\C_n^{k,\weak \dag}$ is a proper subset of  $\C_n^{k, d \dagger}$. In fact, we consider the matrix  
\[
A= \left[\begin{array}{rrrr}
1 &  1 & 0 & 0\\
0 &  0 & -1 & 1  \\
0 &  0 &  0 & 1 \\
0 & 0 & 0 & 0
\end{array}\right].
\]
It is easy to see that $\ind(A)=3$. Moreover, we have 
\[
A^{d,\dag}=A^d=\left[\begin{array}{rrrr}
1 &  1 & -1 & 0\\
0 &  0 & 0 & 0  \\
0 &  0 &  0 & 0 \\
0 & 0 & 0 & 0
\end{array}\right] \quad \text{and} \quad 
A^{\weak, \dag}= \left[\begin{array}{rrrr}
1 &  1 & 0 & 0\\
0 &  0 & 0 & 0  \\
0 &  0 &  0 & 0 \\
0 & 0 & 0 & 0
\end{array}\right].
\]
Therefore, since $A^{d,\dag}=A^d$, from Theorem \ref{class dmp} we have $A\in \C_{n}^{k,d \dagger}$. However, as $A^{\weak,\dag}\neq A^d$, from Theorem \ref{class 4.5} it is clear that $A\notin\C_n^{k,\weak \dag}$.
\end{example} 

\begin{example} \label{example 4} The class $\C_n^{k,\dag \weak}$ is a proper subset of  $\C_n^{k, \dagger d}$. In fact, we consider the matrix  
\[
A= \left[\begin{array}{rrrr}
1 &  -1 & 1 & 0\\
0 &  0 & 0 & 0  \\
0 &  1 &  0 & 0 \\
0 & 0 & 1 & 0
\end{array}\right].
\]
It is easy to see that $\ind(A)=3$. Moreover, we have 
\[
A^{\dag, d}=A^d=\left[\begin{array}{rrrr}
1 &  0 & 1 & 0\\
0 &  0 & 0 & 0  \\
0 &  0 &  0 & 0 \\
0 & 0 & 0 & 0
\end{array}\right] \quad \text{and} \quad 
A^{\dag,\weak}= \left[\begin{array}{rrrr}
1 &  -1 & 1 & 0\\
0 &  0 & 0 & 0  \\
0 &  0 &  0 & 0 \\
0 & 0 & 0 & 0
\end{array}\right].
\]
Therefore, since $A^{\dag, d}=A^d$, from Theorem \ref{class dual dmp} we have $A\in \C_{n}^{k,\dagger d}$. However, as $A^{\dag, \weak}\neq A^d$, from Theorem \ref{class 4.6} it is clear that $A\notin\C_n^{k,\dag \weak}$.
\end{example} 

The other inclusions reflected in Figure 1 are also strict and examples of them can be found at  \cite{FeLeTh3, FeLePrTh, WaLi}.

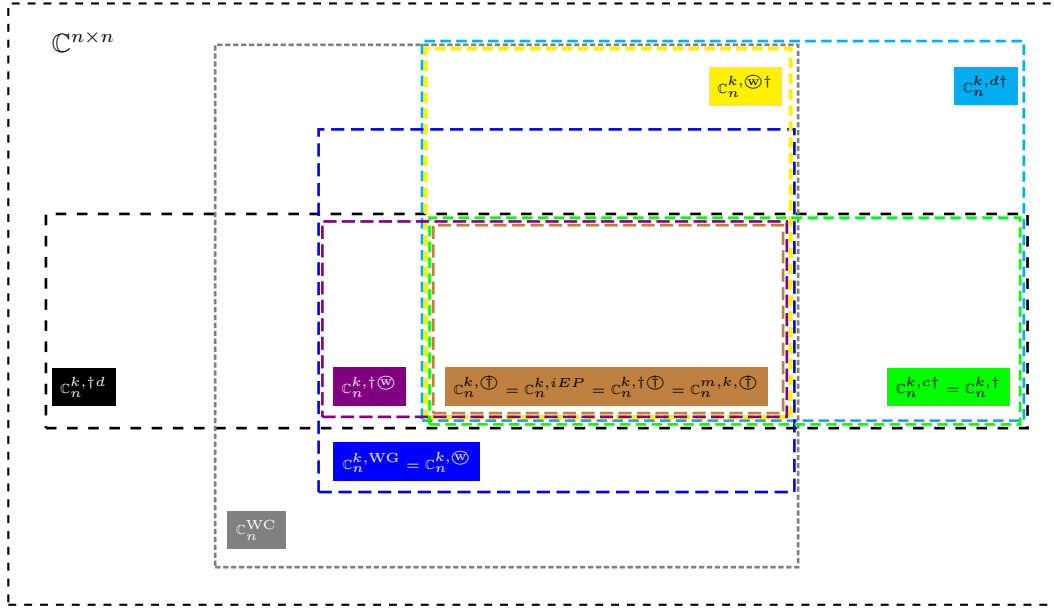
\begin{figure}[!ht]
\centering
\begin{tikzpicture}
\draw[black,thick,dashed] (0,-2) rectangle (14,6);
\node at (1,5.5) {$\Cnn$};
\draw[line width=1pt,dash pattern=on 4pt off 4pt, black] (0.5,0.35) rectangle (13.55,3.2);
\node[fill=black,text=white] at (1,0.9) {\tiny $\C_{n}^{k, \dagger d}$};
\draw[line width=1pt,dash pattern=on 4pt off 2pt,cyan] (5.5,0.45) rectangle (13.5,5.5);
\node[fill=cyan,text=black] at (13,4.9) {\tiny $\C_n^{k, d \dagger}$};
\draw[line width=1pt,dash pattern=on 2pt off 1pt,gray]  (2.75,-1.5) rectangle (10.5,5.45);
\node[fill=gray,text=white] at (3.3,-1) {\tiny $\C_{n}^{ \wc}$};
\draw[line width=1pt,dash pattern=on 5pt off 2.5pt,blue] (4.125,-0.5) rectangle (10.45,4.325);
\node[fill=blue,text=white] at (5.3,-0.1) {\tiny $\C_{n}^{k, \wg}=\C_{n}^{k, \weak}$};
\draw[line width=1.5pt,dash pattern=on 4pt off 2pt,yellow]  (5.55,0.5) rectangle (10.4,5.4);
\node[fill=yellow,text=black] at (9.8,4.9) {\tiny $\C_n^{k,\weak \dag}$};
\draw[line width=1pt,dash pattern=on 4pt off 2pt, green] (5.6,0.4) rectangle (13.45,3.15);
\node[fill=green,text=black] at (12.5,0.9) {\tiny $\C_{n}^{k, c \dagger}=\C_n^{k,\dag}$};
\draw[line width=1pt,dash pattern=on 5pt off 2pt,violet] (4.175,0.5) rectangle (10.35,3.1);
\draw[line width=1pt,dash pattern=on 5pt off 2pt,brown] (5.65,0.55) rectangle (10.3,3.05);
\node[fill=violet,text=white] at  (4.8,0.9) {\tiny $\C_n^{k,\dag \weak}$};
\node[fill=brown,text=black] at (7.95,0.9) {\tiny $\C_{n}^{k, \odagger}=\C_{n}^{k, iEP}=\C_{n}^{k, \dagger \odagger}=\C_{n}^{m,k, \odagger}$};
\end{tikzpicture}
\caption{Matrix classes relationships for $k>1$} 
\label{Figura1}
\end{figure} 

\section*{Conclusions}

It is well known that a square  complex matrix is called EP if it commutes with its Moore-Penrose inverse. 
In \cite{FeLeTh3}, by using generalized inverses in expressions of type $A^kX=XA^k$, where $k$ is the index of $A$,  generalizations of EP matrices was studied, 
which led to the emergence of new matrix classes, such as $k$-index EP, $k$-CMP matrices, $k$-DMP matrices, dual $k$-DMP matrices, and $k$-core EP matrices, investigated in detail by several authors.
This paper presents characterizations of these matrix classes in terms of only commutative equalities of type $AX=XA$, for $X$ being an outer generalized inverse, improving the known results in the literature. We extend widely the results by considering  expressions of the form $A^mX=XA^m$, where  $m$ is an arbitrary positive integer. The
Core-EP decomposition is efficient for investigating the relationships between the different matrix classes  induced by 
generalized inverses.  
Following the current trend in the research of generalized inverses and its applications, we believe that investigation related to matrix classes involving recent generalized  inverses will attract attention, and we describe perspectives for further research: 
1. Extending the matrix classes to Hilbert space operators. 
2. Considering matrix classes induced by weighted generalized inverse such that weighted core-EP inverse, weighted weak group inverse, weighted BT inverse, etc. 


\section*{Declarations} 
 
\noindent {\bf Ethical Approval} 

\noindent Not applicable.  
 
\noindent {\bf Competing interests}
 
\noindent Not applicable.
 
\noindent {\bf Authors' contributions} 

\noindent The whole research was developed by the four authors in close cooperation. Results were obtained by the four authors and improved by themselves in different iterations. N. Thome wrote the Introduction. D.E. Ferreyra and F. Levis wrote Sections 2 and 3. A. Priori and N. Thome wrote Sections 4 and 5. D. Ferreyra wrote Conclusions. F. Levis and A. Priori made Figure 1.

\noindent {\bf  Funding} 
 
\noindent This work was supported by the Universidad Nacional de R\'{\i}o Cuarto (UNRC)  under Grant PPI 18/C559; Universidad Nacional de La Pampa (UNLPam), Facultad de Ingenier\'{\i}a under Grant Resol. Nro. 135/19; Agencia Nacional de Promoci\'on Cient\'{\i}fica y Tecnol\'ogica (ANPCyT) under Grant  PICT 2018-03492; Consejo Nacional de Investigaciones Cient\'{\i}ficas y T\'ecnicas (CONICET) under Grant  PIP 112-202001-00694CO; and Ministerio de Econom\'{\i}a, Industria y Competitividad of Spain (MINECO) under Grant MTM2017-90682-REDT.

\noindent {\bf Availability of data and materials} 

\noindent Not applicable

\end{document}